\documentclass[11pt]{article}
\usepackage{color,latexsym,amsmath,amssymb,amsthm,path,url,enumerate,hyperref, graphicx}

\usepackage{makecell}

\newtheorem{theorem}{Theorem}[section]
\newtheorem{proposition}[theorem]{Proposition}
\newtheorem{lemma}[theorem]{Lemma}

\newtheorem{corollary}[theorem]{Corollary}
\newtheorem{observation}[theorem]{Observation}
\newtheorem{conjecture}[theorem]{Conjecture}

\newtheorem{example}{Example}

\theoremstyle{remark}
\newtheorem{remark}[theorem]{Remark}

\usepackage{dsfont}
\newcommand{\NN}{\ensuremath{\mathbb N}}

\newcommand{\RR}{\ensuremath{\mathbb R}}

\newcommand{\ignore}[1]{}

\newcommand{\ceil}[1]{\left \lceil #1 \right \rceil}

\newcommand{\fancyG}{\ensuremath{\mathcal G}}
\newcommand{\fancyH}{\ensuremath{\mathcal H}}
\def\ex{\operatorname{ex}}
\newcommand{\E}{\mathbb{E}}

\title{Improved bounds on a generalization of Tuza's conjecture}

\author{Abdul Basit\thanks{Department of Mathematics, Iowa State University, Ames IA.\newline \indent\,\,\,\,Email: \texttt{\{abasit, dam1, has, mattds, zerbib\}@iastate.edu}. \newline\indent\,\,\,\,The research of Shira Zerbib was supported by NSF grant DMS-1953929.
\newline\indent\,\,\,\,The research of Daniel McGinnis was supported by NSF grant DMS-1839918 (RTG).} \and 
Daniel McGinnis\footnotemark[1] \and 
Henry Simmons\footnotemark[1] \and
Matt Sinnwell\footnotemark[1] \and
Shira Zerbib\footnotemark[1]
}

\date{}

\begin{document}

\maketitle
\begin{abstract}
For an $r$-uniform hypergraph $H$, let $\nu^{(m)}(H)$ denote the maximum size of a set~$M$ of edges in $H$ such that every two edges in $M$ intersect in less than $m$ vertices, and let $\tau^{(m)}(H)$ denote the minimum size of a collection $C$ of $m$-sets of vertices such that every edge in $H$ contains an element of $C$. The fractional analogues of these parameters are denoted by $\nu^{*(m)}(H)$ and $\tau^{*(m)}(H)$, respectively. Generalizing a famous conjecture of Tuza on covering triangles in a graph, Aharoni and Zerbib conjectured that for every $r$-uniform hypergraph $H$,  $\tau^{(r-1)}(H)/\nu^{(r-1)}(H) \leq \lceil{\frac{r+1}{2}}\rceil$. In this paper we prove bounds on the ratio between the parameters $\tau^{(m)}$ and $\nu^{(m)}$, and their fractional analogues. Our main result is that, for every $r$-uniform hypergraph~$H$,
\[ \tau^{*(r-1)}(H)/\nu^{(r-1)}(H) \le \begin{cases} 
\frac{3}{4}r - \frac{r}{4(r+1)} &\text{for }r\text{ even,}\\
\frac{3}{4}r - \frac{r}{4(r+2)} &\text{for }r\text{ odd.} \\
\end{cases} \]
This improves the known bound of $r-1$.
We also prove that, for every $r$-uniform hypergraph $H$,     $\tau^{(m)}(H)/\nu^{*(m)}(H) \le \operatorname{ex}_m(r, m+1)$, where the Tur\'an number $\operatorname{ex}_r(n, k)$ is the  maximum number of edges in an $r$-uniform hypergraph on $n$ vertices that does not contain a copy of the complete $r$-uniform hypergraph on $k$ vertices. Finally, we prove further bounds in the special cases $(r,m)=(4,2)$ and  $(r,m)=(4,3)$.
\end{abstract}

\section{Introduction}

\subsection{Definitions}  We restrict our attention to {\em $r$-uniform hypergraphs} (or {\em $r$-graphs}), i.e., hypergraphs where all edges are of size $r$.  Denote by $H[X]$ the sub-hypergraph of $H$ induced by $X \subset V(H)$, that is, containing all edges that are contained in $X$. Throughout the paper, we set $V(H) = [n]$, and use the abbreviation $a_1a_2\dots a_k$ to denote the set $\{a_1, a_2, \dots , a_k\} \subseteq V(H)$. We shall also identify $H$ with its edge set $E(H)$.

A {\em matching} in a hypergraph $H$ is a set of disjoint edges. The {\em matching number} $\nu(H)$ is the maximum size of a matching in $H$. A {\em cover} is a set of vertices intersecting all edges of $H$. The {\em covering number} $\tau(H)$ is the minimum size of a cover. Clearly, in an $r$-graph~$H$ we have $ \nu(H) \leq \tau(H) \leq r\nu(H)$. 

More generally, an {\em $m$-matching} $M$ in a hypergraph $H$ is a set of edges such that any two edges in $M$ intersect in less than $m$ vertices. An {\em $m$-cover} $C$ of $H$ is a collection of $m$-sets of vertices such that every edge in $H$ contains at least one member of $C$. So a 1-matching is  a matching and a 1-cover is  a cover. Denote by $\nu^{(m)}(H)$ the {\em $m$-matching number}, that is, the maximum size of an $m$-matching in $H$ and by $\tau^{(m)}(H)$  the {\em $m$-covering number}, namely the minimum size of an $m$-cover of $H$.

For a hypergraph $H$, let $H^{(m)}$ be the hypergraph whose  vertex set is $\binom{V(H)}{m}$ and whose edge set is $\left\{ \binom{e}{m}: e \in H\right\}$. Note that $M \subset H$ is an $m$-matching in $H$ if and only if $\left\{ \binom{e}{m}: e \in M\right\}$ is a matching in $H^{(m)}$. Additionally, a collection $C$ of $m$-sets is an $m$-cover of $H$ if and only if $C$ is a cover of $H^{(m)}$. Thus, $\nu^{(m)}(H) = \nu(H^{(m)})$ and $\tau^{(m)}(H) = \tau(H^{(m)})$.

We will also consider fractional variants of these parameters. A {\em fractional matching} in~$H$ is a function $s: H \rightarrow \RR_{\ge 0}$ satisfying  $\sum_{e \ni v} s(e) \leq 1$ for every $v \in V(H)$. The {\em size} of a fractional matching is $|s|=\sum_{e \in H} s(e)$, and the {\em fractional matching number} of $H$, denoted by $\nu^*(H)$, is the maximum size of a fractional matching of $H$. Similarly, a {\em fractional cover} of $H$ is a function $t: V(H) \rightarrow \RR_{\ge 0}$ satisfying $\sum_{v \in e} t(v) \geq 1$ for every $e \in H$. The {\em size} of a fractional cover is $|t|=\sum_{v \in V(H)} t(v)$, and the {\em fractional covering number} of $H$, denoted by $\tau^*(H)$, is the minimum size of a fractional cover of $H$. 

Determining  the fractional matching and covering numbers is a linear programming problem. The two problems form a dual pair so by the LP duality principle,  $\tau^*(H)=\nu^*(H)$ for every hypergraph $H$. Thus for an $r$-graph $H$ we have \begin{equation}
\label{eq:easylp1}
\nu(H) \leq \nu^*(H) = \tau^*(H) \leq \tau(H) \leq r \nu(H).
\end{equation}

Define a fractional $m$-matching and fractional $m$-cover in $H$ to be a fractional matching and fractional cover in $H^{(m)}$, respectively, and  let $\nu^{*(m)}(H) = \nu^*(H^{(m)})$ and $\tau^{*(m)}(H) = \tau^*(H^{(m)})$. 
Then \eqref{eq:easylp1} implies
\begin{equation*}
\nu^{(m)}(H) \leq \nu^{*(m)}(H) = \tau^{*(m)}(H) \leq \tau^{(m)}(H) \leq \binom{r}{m} \nu^{(m)}(H).
\end{equation*}

\subsection{Tuza's conjecture and its generalizations}
For a finite graph $G$, let $\nu_t(G)$ be the maximum size of a set of edge-disjoint triangles in $G$, and let $\tau_t(G)$ is the minimum size of a set $C$ of edges with the property that each triangle contains a member of $C$. Clearly, $\nu_t(G) \leq \tau_t(G) \leq 3\nu_t(G)$. A famous conjecture of Tuza states:
\begin{conjecture}[Tuza~\cite{tuza90}]
\label{conj:tuza}
For any graph $G$, $\tau_t(G) \le 2\nu_t(G)$.
\end{conjecture}
The inequality is seen to be sharp by taking $G$ to be $K_4$ or $K_5$ (or a disjoint union of these), and is close to sharp in other cases (see e.g.~\cite{bk16, hkt12}). The best known general bound is $\tau_t(G) \leq \frac{66}{23}\nu_t(G),$ due to Haxell~\cite{haxell99}.

Note that if $T(G)$ is the hypergraph whose edges are triples of vertices forming a triangle in $G$, then  $\tau_t(G) = \tau^{(2)}(T(G))$ and $\nu_t(G) = \nu^{(2)}(T(G))$, and thus Tuza's  conjecture states that $ \nu^{(2)}(T(G))\le  2\tau^{(2)}(T(G))$ for every graph $G$.
In this paper we are interested in a generalization of Conjecture~\ref{conj:tuza} proposed by Aharoni and Zerbib~\cite{az20}. They conjectured the same inequality holds for any 3-graph.
\begin{conjecture}[Aharoni-Zerbib~{\cite[Conjecture~1.2]{az20}}]
\label{conj:generaltuza3}
For any 3-graph $H$, $$\tau^{(2)}(H) \le 2\nu^{(2)}(H).$$
\end{conjecture}

They also conjectured that a similar phenomenon should hold much more generally.
\begin{conjecture}[Aharoni-Zerbib~{\cite[Conjecture~1.10]{az20}}]
\label{conj:generaltuzar}
Fix $r \geq 3$. For any $r$-graph $H$, \[ \tau^{(r-1)}(H) \leq \ceil{\frac{r+1}{2}}\nu^{(r-1)}(H).\]

\end{conjecture}

Conjecture~\ref{conj:generaltuzar} is based on a more general conjecture involving functions $h$ and $g$, defined ahead, which states that $h(r, m) = g(r, m)$ for all $r, m$ along with the observation that, for an $r$-graph $H$, if $\nu^{(r-1)}(H) = 1$ then $\tau^{(r-1)} \leq \ceil{\frac{r+1}{2}}$.

 Let $\fancyH_r$ be the family of all $r$-graphs, for some $r\ge 3$. For an integer $2\le m\le r$, we will be interested in the following functions (first defined in \cite{az20}):
\begin{itemize}
    \item $\displaystyle h(r, m) = \sup\left\{ \frac{\tau^{(m)}(H)}{\nu^{(m)}(H)} : H \in \fancyH_r \right\}$,
    \item $\displaystyle g(r, m) = \sup\left\{\tau^{(m)}(H) : H \in \fancyH_r \mbox{ and } \nu^{(m)}(H) = 1 \right\}$,
    \item $\displaystyle h^*(r, m) = \sup\left\{ \frac{\tau^{*(m)}(H)}{\nu^{(m)}(H)} : H \in \fancyH_r \right\}$,
    \item $\displaystyle g^*(r, m) = \sup\left\{\tau^{*(m)}(H) : H \in \fancyH_r \mbox{ and } \nu^{(m)}(H) = 1 \right\}$,
    \item $\displaystyle j^*(r, m) = \sup\left\{ \frac{\tau^{(m)}(H)}{\nu^{*(m)}(H)} : H \in \fancyH_r \right\}.$
\end{itemize}

The following proposition is easily proved from the definitions:
\begin{proposition} For any integers $2\le m\le r$, the following holds.
\begin{enumerate}[(a)]
    \item $g^*(r, m) \le g(r, m) \le h(r, m)$, 
    \item $h^*(r, m) \le h(r, m)$,
    \item $j^*(r, m) \le h(r, m)$.
\end{enumerate}
\end{proposition}

Given a graph $G$, let $H(G, r)$ be the $r$-graph whose vertex set is $V(G)$ and whose edges are cliques of size $r$ in $G$. Let $\fancyG_r \subseteq \fancyH_r$ be the family of all $r$-graphs obtained in this manner.
For each of the functions defined above, we add $\circ$ in the subscript to denote the function where $\fancyH_r$ is replaced by $\fancyG_r$. For example \[ \displaystyle g_\circ(r, m) = \sup\left\{\tau^{(m)}(H) : H \in \fancyG_r \mbox{ and } \nu^{(m)}(H) = 1 \right\}. \]
Observe that $f_\circ(r, m) \leq f(r, m)$, where $f$ is any  of the functions $h,g,h^*,g^*,j^*$.

Using the notation  above,  Conjectures \ref{conj:tuza}, \ref{conj:generaltuza3} and \ref{conj:generaltuzar} can be stated as: 
\begin{conjecture}\label{conj:sum}
\hfill
\begin{enumerate}[(a)]
    \item $h_\circ(3, 2) \leq 2$ {\em (Tuza)},
    \item $h(3, 2) \leq 2$ {\em (Aharoni-Zerbib)},
    \item $h(r, r-1) \leq \ceil{\frac{r+1}{2}}$ {\em  for $r\ge 3$ (Aharoni-Zerbib)}.
\end{enumerate}
\end{conjecture}

In this paper we determine, or give bounds, on the values of the functions $h$, $g$, $h^*$, $g^*$, $j^*, h_{\circ}$, $g_{\circ}$, $h^*_{\circ}$, $g^*_\circ$, $j^*_\circ$,    for certain sets of parameters $r,m$.

\subsection{Our results and  organization of the paper.} In Section~\ref{sec:g-and-h} we focus on the case $m=r-1$. 
That is, we consider the functions $g(r,r-1)$, $h(r,r-1)$, and   their  variants defined above. 
 In \cite{krivelevich95} Krivelevich proved that $h^*_\circ(3,2) \leq 2$.
Aharoni and Zerbib~\cite{az20} generalized this to $r$-graphs, and proved for any $r\ge 3$, \begin{equation}\label{fracrAZ}
  h^*(r,r-1) \leq r-1.
\end{equation}   
This follows from the observation that if $H$ is an $r$-graph then  $H^{(r-1)}$ cannot contain the $r$-uniform  projective plane $\mathbb{P}_r$,  and a theorem of F\"{u}redi~\cite{furedi81} stating that if an $r$-graph $H$ does not contain $\mathbb{P}_r$, then $\tau^*(H) \leq (r-1)\nu(H)$. For an introduction to projective planes, see e.g.~\cite{stinson2007}.

Here we improve the bound (\ref{fracrAZ}) for all $r \geq 4$. Moreover, in the case $r=3$ we give a new (and shorter) constructive proof of (\ref{fracrAZ}), which does not use F\"{u}redi's theorem.

\begin{theorem}
\label{thm:h-r-1-4}
 $h^*(3, 2) = 2$, and $h^*(4, 3) \le \frac{8}{3}$.
\end{theorem}

\begin{theorem}\label{thm:h-r-1}
    For every $r \geq 5$,
\[ h^*(r, r-1) \le \begin{cases} 
\frac{3}{4}r - \frac{r}{4(r+1)} &\text{for }r\text{ even,}\\
\frac{3}{4}r - \frac{r}{4(r+2)} &\text{for }r\text{ odd.} \\
\end{cases} \]
\end{theorem}
For $r \geq 6$, it is also possible to further improve Theorem~\ref{thm:h-r-1}, but the optimization is much more involved, and leads to a negligible gain for large $r$. For example, when $r = 6$, it is possible to obtain an upper bound of $17/4$ (vs. $30/7$ implied by Theorem~\ref{thm:h-r-1}). 

In Section~\ref{sec:(4,2)}, we study the case $(r, m) = (4, 2)$. Some bounds for these parameters and certain families of 4-graphs $H \in \fancyG_4$ are given in Szestopalow{~\cite[Chapter 5]{szestopalow16}}, e.g., when the corresponding graph $G$ is the complete graph, for 4-partite graphs, and planar graphs. In~\cite{az20} it was shown that $g(4,2) = 4$. Completing the picture, we prove 
\begin{theorem}
\label{prop:g(4,2)}\hfill
\begin{enumerate}[(a)]
    \item $g^*_\circ(4, 2) = 2.5$,
    \item $\displaystyle g_\circ(4, 2) = 3$,
    \item $g^*(4, 2) = 3.5$.
\end{enumerate}
\end{theorem}

Observe, in particular, that $g(4,2)\neq g_\circ(4,2)$, unlike the case $r=3$ where it is known that $g(3,2)= g_\circ(3,2)$ (see \cite{az20})  and  conjectured that $h(3,2)= h_\circ(3,2)$.

In \cite{szestopalow17} it was proved that $h_\circ^*(4, 2) \leq 4.5$ (the more general bound  $h^*(4, 2) \leq 4.5$ was obtained in \cite{az20}).
We improve this bound:
\begin{theorem}
\label{thm:hcirc-4-2}
$h^*_\circ(4, 2) \leq 4$
\end{theorem}

Finally, in Section~\ref{sec:j*}, we turn our attention to the function $j^*$. In~\cite{az20}, it was shown that $j^*(3, 2) = 2$ and $j^*(4, 2) \leq 4$. Guruswami and Sandeep~\cite{gs20} made significant progress, proving that $j^*(r, 2) \leq r^2/4$ and $j^*(r, r-1) \leq r/2 + \sqrt{2 r \ln r}$. For the more general case, they give the bound $j^*(r, m) = c(m) \binom{r}{m}$ where $c(m) \rightarrow 1/2 + o(1)$ as $m \rightarrow r-1$. In their paper, they also observed a connection to the hypergraph Tur\'an problem. Here we exploit this connection more explicitly and bound $j^*(r,m)$ in terms of hypergraph Tur\'an numbers.

Let $\ex_r(n, k)$ be the maximum number of edges in an $r$-graph on $n$ vertices that does not contain a copy of $K_{k}^r$, the complete $r$-graph on $k$ vertices. Tur\'an~\cite{turan41} determined the value of $\ex_2(n, k)$ and posed the problem of determining the limit
\[ \pi(k, r) = \lim_{n \rightarrow \infty} \frac{\ex_r(n, k)}{\binom{n}{r}}, \]
for $2 < r < k$. This has proven to be a notoriously hard problem and even the first non-trivial case $r = 3$ and $k = 4$ remains open. For a survey of the problem and related results see~\cite{keevash11, mps11}.

\begin{theorem}
\label{thm:jrm}
$j^*(r,m) \leq \ex_m(r, m+1)$.
\end{theorem}
When $m = r-1$, we obtain the bound of $j^*(r, r-1) \leq r-1$ (this bound also appears in~\cite{gs20}). For $m = 2$, Theorem~\ref{thm:jrm}, together with Mantel's Theorem~\cite{mantel1907}, recovers the bounds $j^*(3, 2) \leq 2$, $j^*(4, 2) \leq 4$ and $j^*(r, 2) \leq \frac{r^2}{4}$. More generally, Theorem~\ref{thm:jrm} together with known results about hypergraph Tur\'an numbers can be used to obtain explicit values. For example, paired with results of Chung and Lu~\cite{cl99}, Markstr\"om~\cite{markstrom09} and  Sidorenko~\cite{sidorenko82}, respectively, we obtain
\begin{itemize}
    \item $\displaystyle \lim_{r \rightarrow \infty} \frac{j^*(r,3)}{\binom{r}{3}} \leq \frac{3 + \sqrt{17}}{12} \approx 0.5936.$
    \item $\displaystyle \lim_{r \rightarrow \infty} \frac{j^*(r,4)}{\binom{r}{4}} \leq \frac{1753}{2380} \approx 0.73655.$
\item $\displaystyle \lim_{r \rightarrow \infty} \frac{j^*(r,m)}{\binom{r}{m}} \leq 1 - \frac{1}{m}$.
\end{itemize}

\section{\texorpdfstring{The case $m=r-1$}{The case m = r - 1}}
\label{sec:g-and-h}

\subsection{Preliminaries}

Let $H \in \fancyH_r$ be an $r$-graph, and $M$ be a maximum $(r-1)$-matching in $H$. We say an edge $e$ in $H$ is of {\em type}-$i$ (with respect to $M$), for $1 \leq i \leq r$, if $e$ intersects exactly $i$ edges of $M$ in $r-1$ vertices each, and $e$ intersects every other edge of $M$ in at most $r-2$ vertices. Let $T_i \subseteq H$ denote the set of edges of type $i$. For $e \in M$, let $H(e) = \{f \in H: |f \cap e| \geq r-1\}$,   $T_i(e) = T_i \cap H(e)$ and $H_i(e)=T_i(e)\cup \{e\}$. For $x\in \binom{e}{r-1}$, if there exists an edge $f\in T_1(e)$ such that $f\cap e=x$, then we say that $x$ is an \textit{indispensable $(r-1)$-set}. When $r = 4$, then we refer to an indispensable $3$-set as an {\em indispensable triple}.

\begin{observation}
\label{obs:friends}
Let $H \in \fancyH_r$, and let $M$ be a maximum $(r-1)$-matching in $H$ with $e\in M$, and let $f,g \in T_1(e)$. Then $|f \cap g| \ge r-1$. In other words, $\nu^{(r-1)}(H_1(e))=1$. 
\end{observation}
\begin{proof}
Assume to the contrary that  $|f \cap g| < r-1$. Then 
$\left(M \setminus \{ e \}\right) \cup \{f, g\}$ is an $(r-1)$-matching of size greater than $|M|$, a contradiction.
\end{proof}

The following lemma gives a structural characterization of $H_1(e)$. 
\begin{lemma}
\label{lem:(r-1)friend}
For $H \in \fancyH_r$,  let $M$ be  a maximum $(r-1)$-matching in $H$. Let  $e \in M$ be an edge such that $T_1(e)\neq \emptyset$. Then one of the following holds:
\begin{enumerate}[(a)]
    \item $\bigcap H_1(e) = r - 1$. If this is true, we set $p(e) = \bigcap H_1(e)$.
    \item There exists a vertex $v \in V(H) \setminus e$ such that $v \in f$ for each $f \in T_1(e)$. If such a vertex exists, we denote it by $v(e)$.
\end{enumerate}
\end{lemma}
\begin{proof}
We may assume $|T_1(e)| > 1$, otherwise the claim is trivial. 
For an edge $f\in T_1(e)$ let $f = x \cup \{v_f\}$,  where $v_f \notin e$.
Let $f, g$ be distinct edges in $T_1(e)$. By Observation~\ref{obs:friends}, if $Y = f \cap g$ then $|Y| = r-1$. 

Assume first $|Y \cap e| = r-1$. Let $h \in T_1(e)$, and assume $h$ does not contain $Y$. Then  $|h \cap f| =r-1$ and $|h \cap g| =r-1$ implies $\{v_f,v_g\} \subset h$, implying $|h \cap e| \le  r-2$, a contradiction.

Otherwise we have $|Y \cap e| < r-1$. Then $|f\cap g|=r-1$ implies  $v_f = v_g$. Similarly, every edge $h \in T_1(e)$ has both $|f\cap h \cap e| < r-1$ and $|g\cap h \cap e| < r-1$, and thus $v_h=v_f$. 
\end{proof}

\begin{remark}
\label{rmk:(r-1)friendexactly1}
If $|T_1(e)| > 1$, then exactly one of (a) or (b) holds.
If (b) holds then $|T_1(e)|$ equals the number of indispensable $(r-1)$-sets (namely, every indispensable $(r-1)$-set belongs to exactly one edge in $T_1(e)$),  and  moreover $|e\cap f \cap g|=r-2$ for distinct edges $f,g \in T_1(e)$. 
\end{remark}

Before proving Theorem~\ref{thm:h-r-1}, we first give the following weaker bound as a warm-up. When $r = 2$, the following proposition recovers the optimal bound $h^*(2, 1) = 1.5$ (see e.g.~\cite{lovasz74}).
\begin{proposition}
\label{prop:h-r-1-weak}
For any $r \geq 2$, \[ \displaystyle \frac{r+1}{2} \leq h^*(r, r-1) \leq \frac{3}{4}r. \]
\end{proposition}

\begin{proof}
To see the lower bound, consider the hypergraph $H = \binom{[r+1]}{r}$. Clearly $\nu^{(r-1)}(H) = 1$. The function $t: \binom{V(H)}{r-1} \rightarrow \mathbb{R}$ with $t(x) = 1/r$ is a fractional $(r-1)$-cover of size $(r+1)/2$, and the function $s: H \rightarrow \mathbb{R}$ with $s(e) = 1/2$ is a fractional $(r-1)$-matching of size $(r+1)/2$. This implies $\tau^*(H) = \nu^*(H) = (r+1)/2$.

For the upper bound, let $H \in \fancyH_r$ be an $r$-graph with $\nu^{(r-1)}(H) = k$, and let $M$ be a maximum $(r-1)$-matching in $H$. Define the function $t: \binom{V(H)}{r-1}\to \mathbb{R}$ by $t(x) = 1/2$ if $x\in \{ \binom{e}{r-1} : e\in M\}$, and $t(x) = 0$ otherwise. Note that $t$ is a fractional $(r-1)$-cover of the type-$i$ edges for $i>1$ and the edges in $M$. It remains to cover type-1 edges, each of which already has weight $1/2$ by $t$.

Let $e\in M$. We define a function $t_e:\binom{V(H)}{r-1}\to \mathbb{R}$ as follows.  
 If $e$ satisfies Lemma~\ref{lem:(r-1)friend}(a), we set $t_e(p(e)) = 1/2$.  If $e$ satisfies Lemma~\ref{lem:(r-1)friend}(b), then we set $t(x) = \frac{1}{2(r-1)}$ for $(r-1)$-sets $x$ satisfying  $|x\cap e|=r-2$ and $v(e)\in x$. For any other $(r-1)$-set we set $t_e(x)=0$. 
 
It is now easy to check that the function $t+\sum_{e\in M} t_e$ is a fractional $(r-1)$-cover of size at most  $\frac{3}{4}r|M|$. 
\end{proof}


\subsection{Proofs of Theorems~\ref{thm:h-r-1-4} and \ref{thm:h-r-1}}

Let $H \in \fancyH_r$ be an $r$-graph, and let $M$ be a maximum $(r-1)$-matching in $H$. For $e \in M$, recall that an indispensable $(r-1)$-set $x \in \binom{e}{r-1}$ is an $(r-1)$-set such that there exists a type-1 edge  $f \in H$ with $f \cap e = x$. We refer to $f$ as a {\em witness} of indispensability of $x$ in $e$. For $0\le i \le r$, let $M_i$ be the set of all edges in $M$ containing exactly $i$ indispensable $(r-1)$-sets. 
For $e,f\in M$ and a type-2 edge $h$, we say that $h$ {\em connects $e$ and $f$} if $|e\cap h|=|f\cap h|=r-1$ (namely, $e,f$ are the two edges in $M$ witnessing the fact that $h$ is type-2).  We define $M^+ = M_3$ if $r=3$, and $M^+ = M_{r-1} \cup M_r$ if $r\geq 4$. Let $M^- = M\setminus M^+$.

\begin{lemma}
\label{lem:specialfriend}
Let $e\in M^+$. Let $f \in M$ and suppose there exists an edge connecting $e$ and~$f$. Then there exists $g \in T_1(e)$ such that:
\begin{enumerate}[(a)]
\item $|g \cap h| < r-1$ for any $h$ connecting $e$ and $f$,
\item $|g \cap a| < r-1$ for each $a \in T_1(f)$.
\end{enumerate}
\end{lemma}
\begin{proof}
Let $h$ be an edge connecting $e$ and $f$. Note first that $|e\cap f|=r-2$ and $e \cap f \subset h \subset e \cup f$. Since $e$ has at least $3$ indispensable $(r-1)$-sets at most two of which contain $e \cap f$, there exists an indispensable $(r-1)$-set $x$ in $e$ such that $e \cap f \nsubseteq x$. Let $g \in T_1(e)$ be a witness for $x$. Then $g$ satisfies the required properties. 
Properties (a) and (b) follow from the fact, proved ahead, that the vertex $v := g \setminus e$ is not contained in $f$ (here we slightly abuse notation and identify the singleton set with the vertex). Indeed $v \notin f$ together with $e \cap f \nsubseteq g$ and, for any $h'$ connecting $e$ and $f$, $h' \subseteq e \cup f$ implies that $|g \cap h'| \leq r-2$, so~(a) is satisfied.
We also have $|g \cap f| = r-3$ implying $|g \cap a| \leq r-2$ for every $a \in T_1(f)$, so (b) is satisfied.

To see that $v \notin f$, let $y$ be an indispensable $(r-1)$-set in $e$ such that $e \cap f \subset y$ and let $a \in T_1(e)$ be a witness for $y$. Such an $(r-1)$-set must exist since $e$ has at least $r-1$  indispensable $(r-1)$-sets at most $r-2$ of which do not contain $e \cap f$. Lemma~\ref{lem:(r-1)friend} applies and by Remark~\ref{rmk:(r-1)friendexactly1}, $a$ contains the vertex $v$. Now, if $v \in f$ then $|a \cap f| \geq r-1$, a contradiction.
\end{proof}

Almost immediately, we obtain the following.

\begin{corollary}
\label{co:notype2M34}
Let $e,f\in M^+$. Then there is no type-2 edge connecting $e$ and $f$. 
\end{corollary}
\begin{proof}
Suppose that $g$ connects $e$ and $f$. By Lemma~\ref{lem:specialfriend}, there exist edges $a \in T_1(e)$ and $b \in T_1(f)$ such that $|a \cap g| < r-1$, $|b \cap g| < r-1$, and $|a \cap b| < r-1$. It follows that $\left(M \setminus \{ e, f \}\right) \cup \{a, b, g\}$ is an $(r-1)$-matching of size greater than $|M|$, a contradiction.
\end{proof}

We call a type-2 edge {\em bad} if it connects an edge $e\in M^+$ and an edge $f\in M^-$. If a type-2 edge is not bad, then it is {\em good}. For an edge $e\in M_i$, let $B(e)$ be the set of all bad type-2 edges connecting $e$ to other edges in $M$.

\begin{corollary}
\label{co:friendbadedgeintersection}
Let $e \in M^-\setminus M_0$, and suppose $g \in B(e)$ connects $e$ and $f \in M+$. Then $|g \cap h| \geq r-1$ for every $h \in T_1(e)$.
\end{corollary}
\begin{proof}
By Lemma~\ref{lem:specialfriend}, there exists an edge $a \in T_1(f)$ such that $|a \cap g| < r-1$ and $|a \cap h| < r-1$ for every $h \in T_1(e)$. If there exists $h \in T_1(e)$ such that $|h \cap g| < r-1$, then $\left(M \setminus \{ e, f \}\right) \cup \{a, g, h\}$ is an $(r-1)$-matching of size greater than $|M|$, a contradiction.
\end{proof}

\begin{lemma}
\label{lem:m2edges}
Let $e\in M^-\setminus(M_0\cup M_1)$, and suppose $g \in B(e)$ connects $e$ and $f \in M^+$. Then the following hold:
\begin{enumerate}[(a)]
\item $v(e) = g \setminus e$
\item If $x$ in an indispensable $(r-1)$-set in $e$, then $e \cap f \nsubseteq x$. 
\end{enumerate}
\end{lemma}
\begin{proof}
Let $x$ be an indispensable $(r-1)$-set in $e$ such that $|x \cap g| = r-2$, which exists because $e\in M_i$ with $i\ge2$, and let $a \in T_1(e)$ be an edge witnessing $x$. By Corollary~\ref{co:friendbadedgeintersection}, $|a \cap g| \geq r-1$ implying that $a$ contains the vertex $(g \setminus e) \in f$. By Lemma~\ref{lem:(r-1)friend}, there exists a vertex $v(e) \notin e$ such that $v$ is contained in every edge of $T_1(e)$. It follows that $v := v(e) = (g \setminus e) \in f$ (as before, we do not distinguish between the singleton set and the vertex).

Now suppose that an edge $b \in T_1(e)$ contains $e \cap f$. Since $v \in b$, we have $|b \cap f| \geq r-1$, a contradiction. It follows that no indispensable $(r-1)$-set in $e$ contains $e \cap f$.
\end{proof}

\begin{corollary}
\label{co:bad23}
If $r = 3$, then there are no bad edges connecting $e \in M_2$ and $f \in M_3$.
\end{corollary}
\begin{proof}
Suppose $g$ connects $e \in M_2$ and $f \in M_3$. This implies $|e \cap f| = 1$. Since $e$ has two indispensable triples, at least one of them contains $e \cap f$, contradicting Lemma~\ref{lem:m2edges}.
\end{proof}

\begin{corollary}
\label{co:bad4}
Suppose $r \geq 4$, and let $e \in M^- \setminus (M_0 \cup M_1)$ such that $B(e) \neq \emptyset$. Let $i$ be such that $e\in M_i$. Then there exist at most $\binom{r}{2} - i(r-1) + \binom{i}{2}$ $(r-1)$-sets such that every edge in $B(e)$  contains at least one of them.
\end{corollary}
\begin{proof}
By Lemma~\ref{lem:m2edges}(a), any edge $g$ that connects $e$ and $f \in M^+$ must contain $e \cap f$ and $v(e)$. Since $e \cap f$ is not contained in any indispensable $(r-1)$-sets (by Lemma~\ref{lem:m2edges}(b)), it suffices to bound the number of $(r-2)$-sets of $e$ that are not contained in any indispensable $(r-1)$-sets of $e$. By inclusion-exclusion, this number is at most $\binom{r}{2} - i(r-1) + \binom{i}{2}$.
\end{proof}

\begin{remark}
If $r\geq 4$, then, for any $e \in M_{r-2}$, there exists an $(r-1)$-set contained in all edges of $B(e)$.
\end{remark}

\begin{lemma}
\label{lem:m1edges}
If $e\in M_1$, then one of the following holds:
\begin{enumerate}[(a)]
    \item all edges in $T_1(e)$ and $B(e)$ share one $(r-1)$-set $w(e)$, or 
    \item $|T_1(e)|=1$. 
\end{enumerate}
\end{lemma}
\begin{proof}
Let $x$ be the indispensable $(r-1)$-set in $e$, and suppose there exists an edge~$g \in B(e)$ that connects $e$ and $f \in M^+$ such that $x \nsubseteq g$. Let $h \in T_1(e)$ be a  witness for $x$. By Corollary~\ref{co:friendbadedgeintersection}, we have $|h \cap g| \geq r-1$ implying that $h$ contains the vertex $g \setminus e$, i.e., $h$ is determined uniquely. It follows that $|T(e)| = 1$.
\end{proof}

We are now ready to prove Theorems~\ref{thm:h-r-1-4} and~\ref{thm:h-r-1}. Since the proofs are similar, we only give details for the more involved proof of Theorem~\ref{thm:h-r-1}. For Theorem~\ref{thm:h-r-1-4}, we simply describe the fractional cover and leave the verification as an exercise.

\begin{proof}[Proof of Theorem~\ref{thm:h-r-1}]
Let $\alpha := \alpha(r) = \frac{r+2}{2(r+1)}$ for $r$ even, and $\alpha := \alpha(r) = \frac{r+3}{2(r+2)}$ for $r$  odd. For $r \geq 5$, let $H \in \fancyH_r$ and let $M$ be a maximum $(r-1)$-matching in $H$. For every~$e\in M$, we define $t_e:\binom{V(H)}{r-1} \to \mathbb{R}$ as follows (if we do not explicitly specify $t_e(x)$ for some $x \in \binom{V(H)}{r-1}$, then $t_e(x) = 0$):
\begin{enumerate}
    \item If $e \in M_0$: Set $t_e(x) = \alpha$ for every $x \in \binom{e}{r-1}$.
    \item If $e \in M_1$: We define functions $t^0_e,\, t^1_e :\binom{V(H)}{r-1}\to \mathbb{R}$ and set $t_e(x) = t_e^0(x) + t_e^1(x)$. 
    
    Set $t_e^0(x) = 1/2$ for every $x \in \binom{e}{r-1}$.
    
    If $e$ satisfies Lemma~\ref{lem:m1edges}(a), then set $t_e^1(w(e)) = 1/2$; otherwise set $t_e^1(x) = \frac{1}{2r}$ for each $x \in \binom{f}{r-1}$ where $f$ is the unique element of $T_1(e)$. 

    \item If $e \in M_i$ for $2\le i \le r-3$:
    Set $t_e(x) = \alpha$ for every $x \in \binom{e}{r-1}$.
    Lemma~\ref{lem:(r-1)friend} (b) applies and by Remark~\ref{rmk:(r-1)friendexactly1} there are exactly~$i$ edges in $T_1(e)$. Observation~\ref{obs:friends} implies that there exist at most $\ceil{i/2}$ $(r-1)$-sets such that every  edge in $T_1(e)$ contains one of these $(r-1)$-sets. For each such $(r-1)$-set~$y$, set $t_e(y) = 1-\alpha$. 
    
    \item If $e \in M_{r-2}$:     We define functions $t^0_e,\, t^1_e,\, t^2_e :\binom{V(H)}{r-1}\to \mathbb{R}$ and set $t_e(x) = t_e^0(x) + t_e^1(x)+t^2_e(x)$. 
    
    Set $t_e^0(x) = 1/2$ for every $x \in \binom{e}{r-1}$.
    Lemma~\ref{lem:(r-1)friend} (b) applies and by Remark~\ref{rmk:(r-1)friendexactly1} there are exactly $r-2$ edges in $T_1(e)$. By Observation~\ref{obs:friends} there exist $\ceil{\frac{r-2}{2}}$ $(r-1)$-sets such that every edge in $T_1(e)$ contains one of them. For each such $(r-1)$-set $y$, set~$t_e^1(y) = 1/2$. 
    
    By Lemma~\ref{lem:m2edges}, if $B(e) \neq \emptyset$ then all edges in $B(e)$ contain an $(r-1)$-set $z$. Set~$t^2_e(z) = \alpha-1/2$.
    
    \item If $e \in M^+$: Set $t_e(x) =1 - \alpha$ for every $x \in \binom{e}{r-1}$. For each $(r-1)$-set $y$ that contains~$v(e)$ and $r-2$ vertices from $e$, set $t_e(y)=\frac{\alpha}{r-1}$.

\end{enumerate}
Let $t = \sum_{e\in M} t_e$. It is easy to check that $t$ covers all edges in $M$, type-1 edges, and type-$i$ edges for $3\leq i\leq r$. Note that $\alpha \geq 1/2$ for all $r$, implying that good type-2 edges are covered. If a type-2 edge is in $B(e)$ for $e \in M_0$, then it is also easily seen to be covered. A type-2 edge in $B(e)$ for $e \in M_1$ gets weight at least $\frac{1}{2}+(1-\alpha)=\frac{2r+1}{2r+2}$ from $(r-1)$-sets contained in edges of $M$, and an additional weight $\frac{1}{2r}$ from $t_e^1(e)$ (which is sufficient by Lemma~\ref{lem:m1edges}). If a type-2 edge is in $B(e)$ for $e \in \cup_{i=2}^{r-3} M_i$, then it receives weight at least~$1$ from $(r-1)$-sets contained in edges of $M$. If a type-2 edge is in $B(e)$ for $e \in M_{r-2}$, then it receives weight at least $1/2+1-\alpha$ from $(r-1)$-sets contained in edges of $M$, and $\alpha-1/2$ from $t_e^2$.

Finally, we have
\[ |t_e| \leq
\begin{cases} 
r\alpha &\text{if }e \in M_0 \\
\frac{r+1}{2}&\text{if }e \in M_1\\
r\alpha+(1-\alpha)\ceil{\frac{r-3}{2}} &\text{if }e \in \bigcup_{i=2}^{r-3}M_i\\
\frac{r}{2}+\frac{1}{2}\ceil{\frac{r-2}{2}}+\alpha-\frac{1}{2}&\text{if }e \in M_{r-2}\\
r(1-\alpha)+\binom{r}{2}\frac{\alpha}{r-1} &\text{if }e \in M+\\
\end{cases}
\]

In each case, it is easy to check that the asserted bound holds.
\end{proof}

\begin{proof}[Proof of Theorem~\ref{thm:h-r-1-4}]
Let $H\in \fancyH_3$, and let $M$ a maximum $2$-matching in $H$. For every~${e\in M}$, we define $t_e:\binom{V(H)}{2} \to \mathbb{R}$ as follows (if we do not explicitly specify $t_e(x)$ for some~$x \in \binom{V(H)}{2}$, then $t_e(x) = 0$):
\begin{enumerate}
    \item If $e \in M_0$: Set $t_e(x) = 2/3$ for every $x \in \binom{e}{2}$.
    \item If $e \in M_1$: We define functions $t^0_e,\, t^1_e :\binom{V(H)}{2}\to \mathbb{R}$ and set $t_e(x) = t_e^0(x) + t_e^1(x)$. 
    
    Set $t_e^0(x) = 1/2$ for every $x \in \binom{e}{2}$.
    
    If $e$ satisfies Lemma~\ref{lem:m1edges}(a), then set $t_e^1(w(e)) = 1/2$; otherwise set $t_e^1(x) = 1/6$ for each $x \in \binom{f}{2}$ where $f$ is the unique element of $T_1(e)$. 
    \item If $e\in M_2$: Set $t_e(x)=1/2$ for every $x\in \binom{e}{2}$. By Observation~\ref{obs:friends}, the two edges in $T_1(e)$ must intersect in a pair $y$. Set $t_e(y)=1/2$.
    \item If $e \in M_3$: Set $t_e(x) = 1/3$ for every $x \in \binom{e\cup \{v(e)\}}{2}$. 
\end{enumerate}
Then $t = \sum_{e \in M} t_e$ is a cover of size at most $2|M|$.

Now let $H \in \fancyH_4$, and let $M$ be a maximum $3$-matching in $H$. For every $e\in M$, we define $t_e:\binom{V(H)}{3} \to \mathbb{R}$ as follows (if we do not explicitly specify $t_e(x)$ for some $x \in \binom{V(H)}{3}$, then $t_e(x) = 0$):
\begin{enumerate}
    \item If $e \in M_0$: Set $t_e(x) = 2/3$ for every $x \in \binom{e}{3}$.
    \item If $e \in M_1$: We define functions $t^0_e,\, t^1_e :\binom{V(H)}{3}\to \mathbb{R}$ and set $t_e(x) = t_e^0(x) + t_e^1(x)$. 
    
    Set $t_e^0(x) = 1/2$ for every $x \in \binom{e}{3}$.
    
    If $e$ satisfies Lemma~\ref{lem:m1edges}(a), then set $t_e^1(w(e)) = 1/2$; otherwise set $t_e^1(x) = 1/6$ for each $x \in \binom{f}{3}$ where $f$ is the unique element of $T_1(e)$. 

    \item If $e \in M_2$:     We define functions $t^0_e,\, t^1_e :\binom{V(H)}{3}\to \mathbb{R}$ and set $t_e(x) = t_e^0(x) + t_e^1(x)$. 
    
    Set $t_e^0(x) = 1/2$ for every $x \in \binom{e}{3}$.
    Lemma~\ref{lem:(r-1)friend} applies and by Remark~\ref{rmk:(r-1)friendexactly1} there are exactly two edges in $T_1(e)$ which, by Observation~\ref{obs:friends}, must intersect in a triple $y$. Set $t_e^0(y) = 1/2$. 
    
    By Corollary~\ref{co:bad4}, if $B(e) \neq \emptyset$ then all edges in $B(e)$ contain a triple $z$. Set $t^1_e(z) = 1/6$.
    
    \item If $e \in M^+$: Set $t_e(x) = 1/3$ for every $x \in \binom{e}{3}$. Since there are at most four edges in~$T_1(e)$ any two of which intersect in a triple, there are two triples $x$ and $y$ such that any element of $T_1(e)$ contains either $x$ or $y$. Set $t_e(x) = t_e(y) = 2/3$.
\end{enumerate}
Then $t = \sum_{e \in M} t_e$ is a cover of size at most $\frac{8}{3}|M|$.
\end{proof}

\section{Fractional 2-covers in 4-graphs}
\label{sec:(4,2)}

\subsection{Preliminaries}

Let $H \in \fancyH_4$, and let $M$ be a maximum $2$-matching in $H$. If we also have $H \in \fancyG_4$ then edges of $H$ correspond to $K_4$'s in the graph $G$ with $V(G) = V(H)$. To avoid confusion, throughout this section by edges we will always mean edges of $H$ and will refer only to subgraphs of $G$.

We say an edge in $H$ is of {\em type-$1$} (with respect to $M$) if it intersects one edge of $M$ in at least two vertices and shares at most one vertex with every other edge of $M$. Let~$T_1 \subseteq H$ denote the set of type-$1$ edges. For $e \in M$, let $H(e) = \{f \in H: |f \cap e| \geq 2\}$ and $T_1(e) = T_1 \cap H(e)$. Note that type-$1$ edges behave similarly to type-$1$ edges in  Section~\ref{sec:g-and-h} giving the following analogue of Observation~\ref{obs:friends}.
\begin{observation}
\label{obs:(4,2)friends}
Let $H \in \fancyH_4$, let $M$ be a maximum $2$-matching in $H$ with $e\in M$, and let $f, g \in T_1(e)$. Then $|f \cap g| \geq 2$.
\end{observation}

For $e \in M$, a pair $p \in \binom{e}{2}$ is {\em indispensable} in $e$ if there exists a type-$1$ edge $f$ with $e \cap f = p$ and we refer to $f$ as a {\em witness} of indispensability of $p$. The following is an immediate consequence of Observation~\ref{obs:(4,2)friends}.
\begin{corollary}
\label{co:(4,2)disjoint}
Let $e \in M$. If there exist two disjoint indispensable pairs $p_1, p_2$ in $e$ with witnesses $f_1, f_2$ respectively, then $f_1 \cap f_2$ is a pair $q = q(p_1, p_2)$ disjoint from $e$.

If we also have $H \in \fancyG_4$, then $G[e \cup q]$ is the graph $K_6$ implying that $H[e\cup q] = \binom{e \cup q}{4}$. 
\end{corollary}

\begin{example}
\label{eg:K6}
Let $H \in \fancyG_4$ be the complete $4$-graph on six vertices, i.e., $H$ is the set of all $K_4$'s contained in $K_6$. Then $\nu^{(2)}(H) = 1$, and $\tau^{*(2)}(H) = \nu^{*(2)}(H) = 2.5$.
\end{example}

Note that Example~\ref{eg:K6} is maximal in the sense that, for a $4$-graph $H'$ with at least seven edges containing a copy of $H$, we have $\nu^{(2)}(H') > 1$. Hence, Corollary~\ref{co:(4,2)disjoint} plays a key role in understanding the structure of $H \in \fancyG_4$.

\subsection{Proof of Theorem~\ref{prop:g(4,2)} (a) and (b)}
The lower bound follows from Example~\ref{eg:K6}.

For the upper bound, let $H \in \fancyG_4$ with $\nu^{(2)}(H) = 1$ and let $e \in H$. Since $\nu^{(2)}(H) = 1$, every edge of $H$ must share at least two vertices with $e$. Note that if $p_1, p_2 \in \binom{e}{2}$ are such that $p_1 \cap p_2 = \emptyset$, then $\{p_1, p_2\}$ covers all edges that share at least three vertices with $e$. It follows that if there is at most one indispensable pair, then $\tau^{(2)}(H) \leq 2$. Indeed, we may take an indispensable pair along with the (unique) pair disjoint from it to be the cover.

By Corollary~\ref{co:(4,2)disjoint}, if $p_1, p_2$ are disjoint indispensable pairs in $e$ then $H[e\cup q(p_1, p_2)] = \binom{e \cup q}{4}$, implying that $H = \binom{e \cup q}{4}$, since otherwise $\nu^{(2)}(H) > 1$. It follows that $\tau^{(2)}(H) \leq 3$ and $\tau^{*(2)}(H) \leq 2.5$.

From here on we assume that there are at least two indispensable pairs, no two of which are disjoint. In particular, there are at most three indispensable pairs. Without loss of generality, let $e = 1234$, and $p_1 = 12, p_2 = 13$ be indispensable pairs with witnesses ${f_1 = 1256}$ and $f_2$ respectively. Since $f_2$ must intersect $f_1$ in at least two vertices, it contains at least one of the vertices $5$ or $6$. Without loss of generality, assume that $f_2 = 135u$ where $u \in V(H)\setminus 24$.
It suffices to consider the following cases.

{\bf Case 1:} $p_1$ and $p_2$ are the only indispensable pairs. Notice that any edge that intersects $e$ in exactly two vertices is covered by the set $C = \{p_1, p_2\}$. Additionally, any edge that intersects $e$ in three vertices and contains the vertex $p_1 \cap p_2 = 1$ is covered by $C$. This implies $\tau^{*(2)}(H) \leq \tau^{(2)}(H) \leq 2$ unless there exists an edge containing the vertices $e\setminus (p_1 \cap p_2) = 234$.

Suppose $g = 234v$, $v \in V(H)\setminus\{1\}$, is an edge in $H$. It is easy to see that  $\{p_1, p_2, 23\}$ is a cover, implying $\tau^{(2)}(H) \leq 3$. Notice that $g$ must intersect both $f_1$ and $f_2$ in two vertices, implying that $v \in f_1 \cap f_2$. If $u \neq 6$ then $f_1 \cap f_2 = 15$ and, hence, $g = 2345$. Observe that the pairs $46$ and $4u$ are not contained in any edge. Indeed, if $46$ is in an edge, then $G[1246]$ and $G[135u]$ are two disjoint $K_4$'s in $G$. If $4u$ is in an edge, then $G[134u]$ and $G[1256]$ are two disjoint $K_4$'s in $G$. 

Let $t: \binom{V(H)}{2} \rightarrow \RR$ be defined as follows: 
\[ t(x) = \begin{cases}
\frac{1}{3} &\mbox{if } x \in \{12, 13, 15, 23, 25, 35\} \\
0 & \mbox{otherwise}
\end{cases}.
\]
Then $t$ is a fractional cover implying $\tau^{*(2)}(H) \leq |t| = 2$. To see that $t$ is a fractional cover, note that any edge that is a witness for $p_1 = 12$ or $p_2 = 13$ must intersect $g$ in at least two vertices and, hence, contains the vertex $5$. Otherwise, if an edge intersects $e$ in three vertices, it must contain $123$ or the vertex $5$. Indeed, if an edge intersects $e$ in~$124$, then in order to intersect $f_2$ in at least two vertices, the edge must contain $5$ (by the discussion above, it cannot contain the pair $4u$). A similar argument shows that an edge that intersects $e$ in $134$ contains the vertex $5$.

Now suppose $u = 6$, then $v \in f_1 \cap f_2 = 156$. If $H$ contains both $2345$ and $2346$ as edges, then $G[123456] = K_6$. As before, it follows that $H = \binom{[6]}{4}$ implying $\tau^{*(2)}(H) \leq 2.5$, and $\tau^{(2)}(H) \leq 3$. We have already dealt with the case where $2345$ is an edge. If $2346$ is an edge, then we may obtain a fractional cover by replacing the vertex $5$ with the vertex~$6$ in the fractional cover~$t$ defined above. The proof follows similarly.


{\bf Case 2:} There exists an indispensable pair $p_3$ such that $p_1, p_2, p_3$ form a triangle in $G$, i.e., $p_3 = 23$. Then $C = \{p_1, p_2, p_3\} = \{12, 13, 23\}$ is a cover. Clearly edges that intersect~$e$ in exactly two vertices are covered. If an edge intersects $e$ in at least three vertices, then it must contain at least two elements of the set $p_1 \cup p_2 \cup p_3 = 123$ and must be covered by $C$. It follows that $\tau^{(2)}(H) \leq 3$.

We now prove that $\tau^{*(2)}(H) \leq 2.5$. Recall that $e = 1234$, $p_1 = 12$ and $p_2 = 13$ have witnesses $f_1 = 1256$ and $f_2 = 135u$ (with $u \in V(H) \setminus 24$). Let $F_1, F_2$ and $F_3$ be the set of witnesses of $p_1, p_2$ and $p_3$ respectively. Note that, for any $i \neq j$, exactly one of the following holds:
\begin{enumerate}[(a)]
    \item There exist witnesses $f_i \in F_i$ and $f_j \in F_j$ with $|f_i \cap f_j| = 2$;
    \item $F_i = \{f_i\}, F_j = \{f_j\}$ and $|f_i \cap f_j| = 3$.
\end{enumerate}
It suffices to consider the following cases:

{\bf Case 2.1: } All pairs $i \neq j$ satisfy (b). Without loss of generality, let $F_1 = \{f_1\} = \{ 1256\}$, $F_2 = \{f_2\} = \{1356\}$ and $F_3 = \{f_3\} = \{2356\}$. Let $t : \binom{V(H)}{2} \rightarrow \RR$ be given by 
\[ t(x) = \begin{cases}
\frac{1}{3} &\mbox{if } x \in \{12, 23, 13\} \\
\frac{1}{6} &\mbox{if } x \in\{14, 15, 16, 24, 25, 26, 34, 35, 36\}\\
0 & \mbox{otherwise}
\end{cases}.\]
Then $t$ is a fractional cover with $|t| = 2.5$. Clearly every edge in $F_1 \cup F_2 \cup F_3$ is covered. If $f \cap e = 123$, then it is covered. Otherwise if $f \cap e \in \{124, 134, 234\}$, then $f$ must contain either the vertex $5$ or the vertex $6$ and so must be covered.

{\bf Case 2.2: } There exists a pair $i \neq j$ satisfying (a). Without loss of generality, suppose $f_1 = 1256 \in F_1$ and $f_2 = 1357 \in F_2$. Now any $f_3 \in F_3$ must satisfy $|f_3 \cap f_1|, |f_3 \cap f_2| \geq 2$ and so either contains the vertex $5$ or the pair $67$. But if $67$ were in some edge of $H$, then $G[1567]$ is a $K_4$ implying $\nu^{(2)}(H) > 1$, a contradiction. It follows that every edge in $F_3$ contains the vertex $5$.

Suppose every edge in $F_1 \cup F_2$ also contains the vertex $5$. Then we define a fractional cover $t: \binom{V(H)}{2} \rightarrow \RR$ as follows: 
\[ t(x) = \begin{cases}
\frac{1}{3} & \mbox{if } x \in \{12, 13, 15, 23, 25, 35\} \\
0 & \mbox{otherwise}
\end{cases}.
\]
To see that $t$ is a cover, note that edges in $F_1, F_2$, and $F_3$ are covered by pairs in $\{12, 15, 25\}$, $\{13, 15, 35\}$, and $\{23, 25, 35\}$, respectively. If $f \cap e = 123$, then clearly it is covered. Otherwise, if $f \cap e \in \{234, 134, 124\}$ then $|f \cap 1256|, |f \cap 1357| \geq 2$ implies that $f$ contains the vertex $5$ and, hence, is covered. Since $|t| = 2$, this implies the assertion.

Assume now that there is an edge in $F_1 \cup F_2$ that does not contain the vertex $5$. Without loss of generality, suppose $f_1' \in F_1$ is such an edge. Since $f_1' \cap (f_2 = 1357) \geq 2$, $f_1'$ contains the vertex $7$. But then $f_1'$ cannot contain the vertex 6 (since $67$ cannot be in an edge of $H$). By definition $f_1'$ cannot contain the vertices $3$ or $4$. Therefore, we may assume $f_1' = 1278$.

For convenience, we recap our assumptions: $e = 1234$, $\{ f_1 = 1256, f_1' = 1278\} \subseteq F_1$, $f_2 = 1357 \in F_2$ and that every edge in $F_3$ contains the triple $235$. We can also assume that the pair $67$ is not contained in any edge of $H$. For the same reason, the pair $58$ cannot be contained in an edge (otherwise $1578 \in H$ implying $\nu^{(2)}(H) > 1$).

Any $f_3 \in F_3$ must be of the form $235u$. Since $|f_3 \cap f_1'| \geq 2$, $f_3$ must contain one of the vertices $7$ or $8$. But by the discussion in the preceding paragraph $f_3$ cannot contain~$8$, implying that $F_3 = \{f_3\} = \{ 2357 \}$. This in turn implies that $f_2 = 1357$ is the unique witness of $p_2$, since $|f_2 \cap f| \geq 2$ for each $f \in \{f_1, f_1', f_3\}$. Let $t: \binom{V(H)}{2} \rightarrow \RR$ be given by: 
\[ t(x) = \begin{cases}
\frac{1}{3} &\mbox{if } x \in \{13, 23\} \\
\frac{2}{3} &\mbox{if } x \in \{35\}\\
1 &\mbox{if } x \in\{12\}\\
0 & \mbox{otherwise}
\end{cases}.
\]
Clearly $F_1, F_2$ and $F_3$ are covered. If $f \cap e \in \{123, 124\}$ then it is covered. If $f \cap e \in \{134, 234\}$ then $f$ cannot satisfy $|f \cap f_1| \geq 2$ and $|f \cap f_1'| \geq 2$, hence such an edge cannot exist. It follows that $t$ is a cover, implying $\tau^{*(2)}(H) \leq \frac{7}{3}$.


{\bf Case 3:} There exists an indispensable pair $p_3$ such that $p_1, p_2, p_3$ form a $K_{1, 3}$ in $G$, i.e., $p_3 = 14$.
Any edge that intersects $e$ in exactly two vertices must contain one of the pairs $\{p_1, p_2, p_3\}$. If an edge intersects $e$ in three vertices and contains the vertex $p_1 \cap p_2 \cap p_3 = 1$, then it is covered by some $p_i$. It follows that $\tau^{(2)}(H) \leq 3$, unless there exists an edge intersecting $e$ in exactly the vertices $234$. Suppose there exists $g$ with $g \cap e = 234$. Since $|g \cap f_1| \geq 2$, $g$ must contain either the vertex~$5$ or the vertex $6$. Without loss of generality, let $g = 2345$. But now any edge that witnesses $p_1$, $p_2$, or $p_3$ must also contain the vertex $5$ (since it must intersect $g$ in at least two vertices). It follows that the set $\{15, 12, 34\}$ is a cover implying $\tau^{(2)}(H) \leq 3$.

To bound the size of the optimal fractional cover, as in Case~2, it suffices to consider the following cases. 

{\bf Case 3.1: } Let $F_1 = \{f_1\} = \{ 1256\}$, $F_2 = \{f_2\} = \{1356\}$ and $F_3 = \{f_3\} = \{1456\}$. But then $G[3456]$ is a $K_4$, implying that $H = \binom{[6]}{4}$.

{\bf Case 3.2: } 
Suppose $f_1 = 1256 \in F_1$ and $f_2 = 1357 \in F_2$. As in Case~2, this implies that every edge in $F_3$ contains the vertex $5$. Observe that the pair $46$ is not contained in an edge otherwise $\{1357, 1246\}$ is a 2-matching in $H$, a contradiction. Similarly the pair~$47$ is not contained in an edge, otherwise $\{1256, 1347\}$ is a 2-matching in $H$. That is, we may assume $f_3 = 1458 \in F_3$.

Now any edge $f_1' \in F_1$ must satisfy $|f_1' \cap f_3| \geq 2$ implying that it must contain the vertex $5$ or the vertex $8$. But if $28 \subset f_1'$, then $\{1357, 2458\}$ is a 2-matching in $H$. It follows that every edge in $F_1$ must contain the vertex $5$. Similarly every edge in $F_2$ must contain the vertex $5$, since if $38$ is contained in an edge then $\{1256, 1378\}$ is a $2$-matching in $H$. Now we define a cover $t: \binom{V(H)}{2} \rightarrow \RR$ as follows: 
\[ t(x) = \begin{cases}
1 & \mbox{if } x \in \{15, 23\} \\
0 & \mbox{otherwise}
\end{cases}.
\]
To see that $t$ is a cover, we note that all type-1 edges contain the pair $15$. If $f \cap e \in \{123, 124, 134\}$, then it contains the vertex $5$, and, hence, the pair $15$. Otherwise if $f \cap e = 234$, then it is covered by $23$. Since $|t| = 2$, this implies the assertion.
\qed

\subsection{Proof of Theorem~\ref{prop:g(4,2)} (c)}

To see the lower bound, let $H \in \fancyH_4$ be the hypergraph with edge set \[ \{1234, 1256, 3456, 1367, 2467, 1457, 2357\}. \]
It is easy to check that $\nu^{(2)}(H) = 1$ and $\tau^{*(2)}(H) = \nu^{*(2)}(H) = 3.5$. The hypergraph $H$ appeared in \cite[Proposition 3.8]{az20}, however, the fractional 2-cover and 2-matching numbers of this hypergraph were not noted. To get some intuition about $H$, let $e = 1234$ and observe that $\binom{e}{2}$ can be decomposed into three pairs of disjoint elements, specifically $\{\{12, 34\}, \{13, 24\}, \{23, 14\}\}$. For each such pair, e.g., $\{12, 34\}$ we add edges $f_1$ and $f_2$ such that $f_1 \cap e = 12$ and $f_2 \cap e = 34$, respectively, and $f_1, f_2$ contain the pair $56$ (which is needed to ensure that $\nu^{(2)}(H) = 1$). Similarly, we add two edges for $\{13, 24\}$, and two edges with for $\{12, 34\}$ while ensuring that $\nu^{(2)}(H) = 1$. 

For the upper bound, let $H \in \fancyH_4$ be a $4$-graph with $\nu_2(H) = 1$, and $e = 1234$ be an edge of $H$. Suppose all edges intersect $e$ in exactly two vertices. We may decompose $\binom{e}{2}$ into three pairs of disjoint elements $M_1 = \{12, 34\}$, $M_2 = \{13, 24\}$, and $M_3 = \{23, 14\}$. By Corollary~\ref{co:(4,2)disjoint}, for each $i \in [3]$, if there are witnesses for both pairs in $M_i$, there is a pair $q_i$ not in $e$ contained in these witnesses. If there is a witness for only one pair in $M_i$,  let $q_i$ be this pair. Let $t: \binom{V(H)}{2} \rightarrow \RR$ be defined as follows: 
\[ t(x) = \begin{cases}
\frac{1}{6} &\mbox{if } x \in \binom{e}{2} \\
\frac{5}{6} &\mbox{if } x \in \{q_1, q_2, q_3\} \\
0 & \mbox{otherwise}
\end{cases}
\]
It is easy to see that $t$ is a fractional cover, implying that $\tau^{*(2)}(H) \leq |t| = 3.5$.

From here on, let $f$ be an edge intersecting $e$ in exactly three vertices. Assume without loss of generality that $f = 1235$. Let $p_1 = 12, p_2 = 23, p_3 = 13$ and $p_4 = 45$, and note that $C = \{p_1, p_2, p_3, p_4\}$ is a cover of $H$ of size $4$. Indeed, for any edge $g$, $|g \cap e|, |g \cap f| \geq 2$ imply the following:
\begin{equation}
\label{eq:g*(4,2)main}
\mbox{Either $g$ contains one of the pairs $p_1, p_2, p_3$, or $g$ contains $p_4$}.
\end{equation}
We may assume that for each $i\in [4]$ there is a nonempty set $E_i$ of edges in $H$ containing~$p_i$ and no other element of $C$, otherwise $\tau(H)\le 3$. 

For $1\leq i\leq 4$, let $f_i\in E_i$. Note that $|f_4 \cap e| \geq 2$ implies $f_4$ contains exactly one of vertices $1$, $2$, or $3$. Without loss of generality, let $f_4 = 3456$. This implies that $E_1 \subseteq \{1246, 1256\}$, since $45 \nsubseteq f_1$ and $|f_1 \cap 3456| \geq 2$. Now if every edge in $E_2 \cup E_3$ contains the vertex $6$, then we obtain a fractional cover $t: \binom{V(H)}{2} \rightarrow \RR$ given by: 
\[ t(x) = \begin{cases}
1 &\mbox{if } x \in \{ 45\} \\
\frac{1}{2} &\mbox{if } x \in \{12, 13, 23, 26, 36\} \\
0 & \mbox{otherwise}
\end{cases}
\]
is a cover of size $|t| = 3.5$.

Therefore, we may assume that there exists an edge in $E_2 \cup E_3$ that does not contain the vertex 6. Without loss of generality, let $f_2 \in E_2$ be such an edge. Since $45 \nsubseteq f_2$ and $|f_2 \cap f_4| \geq 2$, $f_2$ must contain exactly one of the vertices $4$ or $5$. We assume that $f_2 = 2347$; the case when $f_2 = 2357$ can be dealt with by interchanging the roles $e$ and $f$ and relabelling the vertices appropriately. Since any edge $f_1 \in E_1$ satisfies $|f_1 \cap f_2| \geq 2$, this implies that $E_1 = \{1246\}$.

 For convenience, we reiterate that we may assume that $H$ contains the edges $e = 1234, f = 1235, f_2 = 2347$ and $f_4 = 3456$, and that $E_1 = \{f_1\} = \{ 1246 \}$. As a consequence, we obtain that every edge $f_2' \in E_2$ must contain the vertex $4$ or the vertex $6$ (since $|f_2' \cap f_1|, |f_2' \cap f_4| \geq 2$), and every edge $f_3 \in E_3$ must contain the vertex $4$ or the pair $67$ (since $|f_3 \cap f_1|, |f_3 \cap f_2| \geq 2$).

If every edge in $E_2$ contains the vertex $4$, a fractional cover $t: \binom{V(H)}{2} \rightarrow \RR$ is given by: 
\[ t(x) = \begin{cases}
1 &\mbox{if } x \in \{ 45\} \\
\frac{1}{2} &\mbox{if } x \in \{12, 13, 16, 23, 34\} \\
0 & \mbox{otherwise}
\end{cases}
\]
implying $\tau^{*(2)}(H) \leq |t| = 3.5$.

We may now assume that there is $f_2' \in E_2$ such that $f_2' = 236w$ with $w \in V(H) \setminus 14$.

{\bf Case 1:} Suppose $1245$ is not an edge in $H$. Then $t: \binom{V(H)}{2} \rightarrow \RR$ given by: 
\[ t(x) = \begin{cases}
\frac{1}{2} &\mbox{if } x \in \{13, 23, 24, 26, 34, 36, 45\} \\
0 & \mbox{otherwise}
\end{cases}\]
is a cover of size $|t| = 3.5$. Indeed any edge that contains two or more pairs in $C$ is covered (since $1245 \notin H$). Also, edges in $E_1$ are covered by $\{24, 26\}$, edges in $E_2$ are covered by $\{23,34,36\}$, and edges in $E_3$ are covered by $\{ 13, 34, 36\}$. Every edge in $E_4$ must contain one of the vertices $1, 2$ or $3$, since it must intersect $1234$ in a pair. But it cannot contain~$1$, since otherwise it does not intersect $f_2'$ in a pair, a contradiction. If it contains 2 or 3 then it is covered by $\{ 45, 24, 34 \}$.  


{\bf Case 2:} Suppose $g = 1245$ is an edge in $H$. Since $|f_2'\cap g|\geq 2$, $f_2'=2356$. Then $t: \binom{V(H)}{2} \rightarrow \RR$ given by: 
\[ t(x) = \begin{cases}
1 &\mbox{if } x \in \{23, 45\} \\
\frac{1}{2} &\mbox{if } x \in \{12, 13, 16\} \\
0 & \mbox{otherwise}
\end{cases}
\]
is a fractional cover of size $|t| = 3.5$. Indeed any edge that contains at least two elements of~$C$ is covered. Any edge that contains $p_2$ or $p_4$ is also covered. Edges in $E_1$ are covered by $\{12, 16\}$. Recall that edges in $E_3$ must contain the vertex $4$ or the pair $67$. But any~$f_3 \in E_3$ must also satisfy $|f_3 \cap 1245| \geq 2$ implying that $f_3$ must contain the vertex $4$. Finally, $|f_3 \cap 2356| \geq 2$ implies that $E_3 = \{ 1346 \}$ and so is covered by $\{13, 16\}$.

This concludes the proof of the theorem. 
\qed


\subsection{Proof of Theorem~\ref{thm:hcirc-4-2}}

Let $H \in \fancyG_4$ be a $4$-graph, and let $M$ be a maximum $2$-matching in $H$.  Define the function $t: \binom{V(H)}{2}\to \mathbb{R}$ by $t(x) = 1/2$ if $x\in \{ \binom{e}{2} : e\in M\}$, and $t(x) = 0$ otherwise. Note that $|t| = 3|M|$ and that $t$ is a fractional $2$-cover of all edges in $H \setminus T_1$. It remains to cover type-1 edges, each of which already has weight $1/2$ by $t$. To finish the proof, it suffices to show that, for any $e \in M$, edges in $T_1(e)$ can be covered using additional weight at most~1. This follows from the fact, proved ahead, that there exist two pairs such that every edge in $T_1(e)$ contains at least one of them. Indeed, assigning weight $1/2$ to both these pairs suffices to cover all edges in $T_1(e)$.

Suppose that there are two disjoint indispensable pairs $p_1, p_2$ with witnesses $f_1$, $f_2$. By Corollary~\ref{co:(4,2)disjoint}, $f_1 \cap f_2 = q$ is a pair disjoint from $e$ and $H[e \cup q] = \binom{e \cup q}{4}$. In particular, every pair contained in $e$ in indispensable and is witnessed by an edge that contains $q$.  Suppose $f \in T_1(e)$ does not contain $q$ and set $p = f \cap e$. Then the edge $f' \in F(e)$ which witnesses $e \setminus p$ and contains $q$ shares at most one vertex with $f$, contradicting  Observation~\ref{obs:(4,2)friends}. It follows that every edge in $T_1(e)$ contains $q$. 

Since out of any four indispensable pairs two must be disjoint, from here on, we may assume that there are exactly three indispensable pairs. Without loss of generality, let $e = 1234$, $p_1 = 12$, $p_2 = 13$ and $p_3$ be indispensable pairs. For $i \in [3]$, denote by $F_i$ the set of edges witnessing $p_i$.  We may also assume that, for $i \neq j$, edges in $F_i \cup F_j$ do not contain a common pair $p$, otherwise $p \subseteq \bigcap_{f \in F_i \cup F_j} f$ and all edges in $T_1(e)$ contain either $p$ or $p_k$, where $k \notin \{i,j\}$. In particular, for $i \neq j$, we have $|\bigcap_{f \in F_i \cup F_j} f| = 1$ implying that there are at least three edges in $F_i \cup F_j$ (any two edges in $T_1(e)$ intersect in at least two vertices). Also since, by assumption, both $F_i$ and $F_j$ are non-empty, we have that there exist edges $f_i \in F_i$ and $f_j \in F_j$ such that $|f_i \cap f_j| = 2$.

The above discussion lets us assume, without loss of generality, that $f^1_1, f^2_1 \in F_1$ and $f_2 \in F_2$ with $f^1_1 = 1256$, $f_2 = 1367$ and $|f^1_1 \cap f^2_1 \cap f_2| = 1$. Since $|f^2_1 \cap f_2| \geq 2$, $f^2_1$ must contain the vertex $7$, i.e., $f^2_1 = 127u$, $u \in V(H) \setminus 346$. We may also assume that $u \neq 5$, since otherwise $G[2567]$ is a $K_4$ implying $2657 \in H$. But $2657$ intersects every edge of $M$ in at most one vertex (a pair contained in an edge of $T_1(e)$ cannot be contained in any other edge in $M$) implying that $M \cup \{ 2567\}$ is a matching of size greater than $|M|$, a contradiction to the maximality of $M$. From here on, we may assume that $f^2_1 = 1278$.

In the preceding paragraph, we use the following observation about $H$ and the corresponding graph $G$. Since it is used again throughout the proof, we make it explicit. Let $p \subseteq f$ be a pair where $f\in T_1(e)$. Any edge in $M\setminus e$ (a $K_4$ in $G$) can share at most one vertex with $p$ (which corresponds to an edge of $G$). Hence, if there exists an edge $e'$ containing at most one vertex in $e$ such that all pairs in $e'$ are contained in an edge of $T_1(e)$ then $M \cup e'$ is a matching of size greater than $|M|$, a contradiction.

For each $u \in 5678$, the pair $1u$ is contained in $f_1^1$ or $f_1^2$. The pairs $56, 67$ and $78$ are contained in $f_1^1, f_2$ and $f_1^2$ respectively. By the preceding paragraph, it follows that neither of the pairs $57$ or $68$ can be contained in an edge of $H$ (otherwise $1567$ or $1268$ is an edge of $H$). Let $f_3$ be a witness for $p_3$, and note that $|f_3 \cap f| \geq 2$ for each $f \in \{f_1^1, f_1^2, f_2\} = \{1256, 1278, 1367\}$ implies that $f_3$ must contain one of the pairs in $\{57,67,68\}$. We have already argued that $57$ and $68$ cannot be contained in edges of $H$, so $f_3$ is determined uniquely. It suffices to consider the following cases:

{\bf Case 1: } $p_3 = 23$ implying that $f_3 = 2367$. This implies that every edge in $F_2$ contains the pair $36$. Indeed, let $f \in F_2$ be an edge that does not contain $36$, i.e., $f$ does not contain~$6$. Since $|f \cap 1256| \geq 2$ and $|f \cap 1278| \geq 2$, $f$ must contain the vertices $5$ and~$8$ ($57$ cannot be contained in an edge). That is, $f = 1358$. But then $f$ and $f_3 = 2367$ are edges in $T_1(e)$ that share less than two vertices, contradicting Observation~\ref{obs:(4,2)friends}.

We have shown that every edge in $F_2 \cup F_3$ contains the pair $36$, and so every edge in $T_1(e)$ contains at least one of the pairs $12$ or $36$.

{\bf Case 2: } $p_3 = 14$ implying that $f_3 = 1467$. But then $3467 \in H$ (since $367 \in f_2$ and $34 \in e$) and $1278$ are edges in $T_1(e)$ that share less than two vertices, a contradiction.
\qed


\section{\texorpdfstring{The function $j^*$}{The function j*}}
\label{sec:j*}

\subsection{The hypergraph Tur\'an Problem}

Recall that $\ex_r(n, k)$ is the maximum number of edges in an $r$-graph on $n$ vertices that does not contain a copy of $K_{k}^r$, the complete $r$-graph on $k$ vertices.
Let $T(n, k, r)$ be the minimum number of edges in an $r$-graph $H$ on $n$ vertices such that any subset of $k$ vertices contains at least one edge of $H$. Note that if $H$ is an $r$-graph on $[n]$ such that any subset of $k$ vertices contains at least one edge of $H$, and $G$ is the complement of $H$, i.e., $G=\{e \in \binom{[n]}{r}: e \not \in H\}$, then $G$ is an $r$-graph on $[n]$ that does not contain a copy of~$K^r_k$; and vice-versa. It follows that $T(n,k,r) + \ex_r(n,k) = \binom{n}{r}.$ Let \[ t(k,r) = \lim_{n \rightarrow \infty} \frac{T(n,k,r)}{\binom{n}{r}}, \quad \mbox{and} \quad \pi(k,r) = \lim_{n \rightarrow \infty} \frac{\ex_r(n,k)}{\binom{n}{r}}.\]
Then we have $ t(k,r) + \pi(k, r) = 1.$ The ratio $T(n,k,r)/\binom{n}{r}$ is non-decreasing (see e.g.~\cite{sidorenko97}), hence the limits above exist and, for any $n$,
\begin{equation}
\label{eq:limbound}
 T(n, k, r) \leq t(k, r) \binom{n}{r}, \quad \mbox{and} \quad \ex_r(n, k)\geq \pi(k, r) \binom{n}{r}.
\end{equation} 

Given an $r$-graph $H$, a {\em $K_{k}^{r}$-cover} of $H$ is a set $C$ of edges such that every $K_{k}^{r}$ in $H$ contains at least one edge of $C$. The number $T(n, k, r)$ is precisely the size of the smallest $K_{k}^{r}$-cover of $H = K_n^r$. Let $\Tilde{t}(r) = \inf \{t \in \RR: \forall\, r\mbox{-graphs}\,H\,\exists \mbox{ a $K_{r+1}^{r}$-cover of size} \leq t |H|\}$. Clearly
\begin{equation}
\label{eq:tildetbound}
t(r+1, r) \leq \Tilde{t}(r), \quad \mbox{and} \quad \pi(r+1, r) \geq 1 - \Tilde{t}(r).
\end{equation}
We show that some constructions that give bounds on $t(r+1, r)$ also give bounds for $\Tilde{t}(r)$. It is well known that $\Tilde{t}(2) \leq 1/2$.
Tur\'an~\cite{turan41} showed that $t(4,3) \leq 4/9$, and conjectured that this was optimal. Based on Tur\'an's construction, we have the following.
\begin{lemma}
\label{lem:cover43}
For any 3-graph $H$, there exists a $K_{4}^3$-cover of size at most $\frac{4}{9}|H|.$ That is~$\Tilde{t}(3) \leq 4/9.$
\end{lemma}
\begin{proof}
Let $V_0 \cup V_1 \cup V_2 = V(H)$ be a uniform random partition of $V(H)$, and set $V_3 = V_0$. Specifically, each $u \in V(H)$ is in $V_i$, $i \in [3]$, with probability $1/3$ independent of other vertices. Let $C$ be the collection of edges that have all three vertices in one part, or, for some $0 \leq i \leq 2$,  have two vertices in $V_i$ and one vertex in $V_{i+1}$. It is easy to check that $H \setminus C$ is $K_4^3$-free. For an edge $e \in H$, the probability that $e \in C$ is $4/9$. It follows that $\E|C| = \frac{4}{9} |H|$, which implies the claim.
\end{proof}

Extending the same idea to $4$-graphs gives:
\begin{lemma}
\label{lem:cover54}
For any 4-graph $H$, there exists a $K_{5}^4$-cover of size at most $\frac{3}{8}|H|.$ That is~$\Tilde{t}(4) \leq 3/8.$
\end{lemma}
\begin{proof}
Since the proof is very similar to that of Lemma~\ref{lem:cover43}, we only give a sketch. Let $V_0 \cup V_1 \cup V_2 \cup V_3 = V(H)$ be a uniform random partition of $V(H)$, and set $V_{4} = V_0$ and $V_5=V_1$. Let $C$ be the collection of edges that satisfy one of the following:
\begin{itemize}
    \item all four vertices are in one part,
    \item there is one vertex in each part,
    \item For $i \neq j$, $|e \cap V_i| = 2$ and $|e \cap V_j| = 2$,
    \item For some $0 \leq i \leq 3$, $|e \cap V_i| = 3$ and $|e \cap V_{i+1}| = 1$,
    \item For some $0 \leq i \leq 3$, $|e \cap V_i| = 3$ and $|e \cap V_{i+2}| = 1$.
\end{itemize}
It is easy to check that $C$ is a cover and that $\E|C| = \frac{3}{8}|H|$ implying the assertion.
\end{proof}

For larger $r$, we use the following construction of Frankl and R{\"o}dl~\cite{fr85}. There are constructions that give better bounds, but this suffices for our purposes.
\begin{lemma}
\label{lem:hypergrahcover}
For any $l \in \NN$ and any $r$-graph $H$, there exists a $K_{r+1}^r$-cover of size at most \[ \left[\frac{1}{l} + \left(1 - \frac{1}{l}\right)^{r}\right] \left|H\right|. \]
\end{lemma}
\begin{proof}
Let $A_0, \dots, A_{l-1}$ be a random partition of $V(H)$, where each vertex is placed into one of $A_0, \dots, A_{l-1}$ with probability $1/l$ independently of other vertices. For $B \subset V(H)$, we define 
\[ d(B) = \left|\left\{ i \in \{0, \dots, l-1\}: B \cap A_i = \emptyset \right\}\right|, \]
and 
\[ w(B) = \sum_{i = 0}^{l-1} i |B \cap A_i|.\]
For $0 \leq j \leq l-1$, let $\mathcal{C}_j$ be the family
\[ \mathcal{C}_j = \left\{ e \in H: \left( w(e) + j\right) \bmod l  \in \{0, \dots, d(e)\} \right\}. \]
Then for every $0 \leq j\leq l-1$, the set $\mathcal{C}_j$ covers all copies of $K_{r+1}^{r}$ in $H$. To see this, let $U \in \binom{V(H)}{r+1}$ be such that $H[U]$ is $K_{r+1}^r$. Since there are $l - d(U)$ indices such that $U \cap A_i \neq \emptyset$, at least one such index~$i$ must be in 
\[ (w(U) + j) \bmod l, ~(w(U) + j - 1) \bmod l,~\dots~,~(w(U) + j - d(U))  \bmod l. \] 
Let $x \in U \cap A_i$ and $e = U \setminus x$.  Now, since $w(e) \equiv w(U) - i \pmod l$ and $d(e) \geq d(U)$, we have $0 \leq (w(e) + j) \bmod l \leq d(e) $ implying $e \in \mathcal{C}_j$.

Since each $e \in H$ belongs to exactly $d(e) + 1$ of the families $\mathcal{C}_0, \dots, \mathcal{C}_{l-1}$, 
\[ \sum_{j = 0}^{l-1} |\mathcal{C}_j| = \sum_{e \in H} (d(e) + 1) = |H| + \sum_{i = 0}^{l-1}|\mathcal{A}_i|,\]
where $\mathcal{A}_i = \{ e \in H : e \cap A_i = \emptyset\}$. Note that
\[ \E|\mathcal{A}_i| =  \sum_{e \in H}P(\{e \cap A_i = \emptyset\}) =  \left|H\right| \left(1 - \frac{1}{l}\right)^{r}, \]
which gives \[ \E\sum_{j = 0}^{l-1}|\mathcal{C}_j| = \left[1 + l\left(1 - \frac{1}{l}\right)^{r}\right] \left|H\right|. \]
It follows that there exists a collection of covers $(\mathcal{C}_j : 0 \leq j \leq l-1)$ whose sizes sum up to the expected value. Since we may pick the smallest of these, there exists a cover of size at most \[ \left[\frac{1}{l} + \left(1 - \frac{1}{l}\right)^{r}\right] \left|H\right|.\]
\end{proof}

Setting $l = (r/\ln r)\left(1 + o(1)\right)$ in Lemma~\ref{lem:hypergrahcover} gives $\Tilde{t}(r) \leq \frac{\ln r + O(1)}{r}$. For our purposes the following very weak, but exact, bound suffices.

\begin{corollary}
\label{co:hypergrahcover}
For every $r \geq 5$, and any $r$-graph $H$, there exists a $K_{r+1}^r$-cover of size at most $\frac{113}{243} \left|H\right|.$ That is $\Tilde{t}(r) \leq \frac{113}{243} \approx 0.4650$.
\end{corollary}
\begin{proof}
Setting $l = 3$ in Lemma~\ref{lem:hypergrahcover} implies that, for every $r \geq 5$, there is a $K_{r+1}^{r}$-cover of size at most
\begin{equation}
\label{eq:r5bound}
\left[\frac{1}{3} + \left(\frac{2}{3}\right)^r\right]\left|H\right|.    
\end{equation}
For $r = 5$, we obtain $\Tilde{t}(r) \leq \frac{113}{243}$. Clearly the function in \eqref{eq:r5bound} is decreasing in $r$, implying the assertion.
\end{proof}

Lemmas \ref{lem:cover43} and \ref{lem:cover54}, and Corollary~\ref{co:hypergrahcover} together imply the following.

\begin{corollary}\label{tbound}
For $r\geq 2$, $\Tilde{t}(r)\leq 1/2$.
\end{corollary}

\subsection{Proof of Theorem~\ref{thm:jrm}}
\label{sec:jrmproof}
For convenience, we let $\ex = \ex_m(r, m+1)$ for the rest of this section. Suppose, for contradiction, that the assertion is false and let $H \in \fancyH_r$ be a minimal counterexample. That is, suppose $H$ satisfies $\tau^{(m)}(H) > \ex \nu^{*(m)}(H) $, and every subhypergraph $H'$ of $H$ satisfies $\tau^{(m)}(H) \leq \ex \nu^{*(m)}(H)$. Let $g$ be a minimum fractional $m$-cover, and $f$ be a maximum fractional $m$-matching respectively.

Let $U$ be the collection of $m$-sets $u \in \binom{V(H)}{m}$ for which $g(u)>0$. By complementary slackness, 
\begin{equation}
\label{eq:slackness}
|U| = \sum_{u\in U} 1 = \sum_{u\in U} \sum_{e \in H\,:\, e \supset u} f(e) = \binom{r}{m} \nu^{*(m)}(H).
\end{equation}

First assume that there exists $u \in U$ with $g(u) \geq 1/\ex$, and let $H'$ be the hypergraph obtained from $H$ by removing all edges containing $u$. Then we have $$ \tau^{(m)}(H) \leq \tau^{(m)}(H') + 1, \quad \text{and} \quad \tau^{*(m)}(H) \geq \tau^{*(m)}(H') + 1/\ex.$$ It follows that
\begin{align*}
\tau^{(m)}(H) & \leq \tau^{(m)}(H') + 1\\
 & \leq \ex \tau^{*(m)}(H') + 1\\
 & \leq \ex \left(\tau^{*(m)}(H) - \frac{1}{\ex}\right) + 1\\
 & = \ex\tau^{*(m)}(H),
\end{align*}
contradicting the assumption on $H$.

We may now assume $g(u) < 1/\ex$ for each $u \in U$. In particular, every edge $e \in H$ contains at least $\ex + 1$ of the $m$-sets in $U$. This implies that, if we consider $U$ as an $m$-graph on $V(H)$ then every edge of $H$ corresponds to an $r$-set of vertices containing a copy of $K_{m+1}^m$. By the definition of $\Tilde{t}(m)$, for every $\epsilon>0$
there exists a $K_{m+1}^{m}$-cover $C$ of~$U$ such that $|C| \leq (\Tilde{t}(m)+\epsilon)|U|$. Note also that every edge of $H$ contains at least one $m$-set in $C$, i.e., $C$ is an $m$-cover for $H$. It follows that 
\begin{align*}
\tau^{(m)}(H)& \leq (\Tilde{t}(m)+\epsilon) |U| = (\Tilde{t}(m)+\epsilon) \binom{r}{m} \nu^{*(m)}(H)\\
&\leq (1-\Tilde{t}(m)+\epsilon) \binom{r}{m} \nu^{*(m)}(H) \leq \ex_{m}(r, m+1)\nu^{*(m)}(H)+\epsilon\binom{r}{m} \nu^{*(m)}(H), 
\end{align*}
where the second to last inequality follows from Corollary \ref{tbound}, and the last inequality follows from \eqref{eq:limbound}, \eqref{eq:tildetbound}. Taking $\epsilon \rightarrow 0$, we obtain the result.
\qed

\section*{Acknowledgement} 

We are grateful to the referees who made many helpful suggestions to improve the clarity of our presentation. 

This paper was written as part of the Iowa State Mathematical Research Teams. We are grateful to the Department of Mathematics at Iowa State University for supporting this project.



\begin{thebibliography}{99}

\bibitem{az20}
R.~Aharoni and S.~Zerbib.
\newblock A generalization of {T}uza's conjecture.
\newblock {\em Journal of Graph Theory}, 94(3):445--462, 2020.

\bibitem{bk16}
J.~Baron and J.~Kahn.
\newblock Tuza's conjecture is asymptotically tight for dense graphs.
\newblock {\em Combinatorics, Probability and Computing}, 25(5):645--667, 2016.

\bibitem{cl99}
F.~Chung and L.~Lu.
\newblock An upper bound for the {T}ur{\'a}n number $t_3(n, 4)$.
\newblock {\em Journal of Combinatorial Theory, Series A}, 87(2):381--389,
  1999.

\bibitem{fr85}
P.~Frankl and V.~R{\"o}dl.
\newblock Lower bounds for {T}ur{\'a}n's problem.
\newblock {\em Graphs and Combinatorics}, 1(1):213--216, 1985.

\bibitem{furedi81}
Z.~F{\"u}redi.
\newblock Maximum degree and fractional matchings in uniform hypergraphs.
\newblock {\em Combinatorica}, 1(2):155--162, 1981.

\bibitem{gs20}
V.~Guruswami and S.~Sandeep.
\newblock Approximate hypergraph vertex cover and generalized {T}uza's
  conjecture.
\newblock {\em Proceedings of the 2022 Annual ACM-SIAM Symposium on Discrete Algorithms (SODA)}, 927--944, 2022. 
\bibitem{haxell99}
P.~Haxell.
\newblock Packing and covering triangles in graphs.
\newblock {\em Discrete mathematics}, 195(1-3):251--254, 1999.

\bibitem{hkt12}
P.~Haxell, A.~Kostochka, and S.~Thomass{\'e}.
\newblock A stability theorem on fractional covering of triangles by edges.
\newblock {\em European Journal of Combinatorics}, 33(5):799--806, 2012.

\bibitem{keevash11}
P.~Keevash.
\newblock Hypergraph {T}ur\'an problems.
\newblock {\em Surveys in combinatorics}, 392:83--140, 2011.

\bibitem{krivelevich95}
M.~Krivelevich.
\newblock On a conjecture of {T}uza about packing and covering of triangles.
\newblock {\em Discrete Mathematics}, 142(1-3):281--286, 1995.

\bibitem{lovasz74}
L.~Lov{\'a}sz.
\newblock Minimax theorems for hypergraphs.
\newblock In {\em Hypergraph {S}eminar}, pages 111--126. Springer, 1974.

\bibitem{mantel1907}
W.~Mantel.
\newblock Problem 28.
\newblock {\em Wiskundige Opgaven}, 10(60-61):320, 1907.

\bibitem{markstrom09}
K.~Markstr{\"o}m.
\newblock Extremal hypergraphs and bounds for the {T}ur{\'a}n density of the
  4-uniform ${K}_5$.
\newblock {\em Discrete mathematics}, 309(16):5231--5234, 2009.

\bibitem{mps11}
D.~Mubayi, O.~Pikhurko, and B.~Sudakov.
\newblock Hypergraph {T}ur{\'a}n problem: {S}ome open questions.
\newblock In {\em AIM workshop problem lists, manuscript}, 2011.

\bibitem{sidorenko82}
A.~Sidorenko.
\newblock The method of quadratic forms in a combinatorial problem of
  {T}ur{\'a}n.
\newblock {\em Moscow Univ. Math. Bull}, 37:1--5, 1982.

\bibitem{sidorenko97}
A.~Sidorenko.
\newblock Upper bounds for {T}ur{\'a}n numbers.
\newblock {\em Journal of Combinatorial Theory, Series A}, 77(1):134--147,
  1997.

\bibitem{stinson2007}
D.~Stinson.
\newblock {\em Combinatorial designs: constructions and analysis}.
\newblock Springer Science \& Business Media, 2007.

\bibitem{szestopalow16}
M.~Szestopalow.
\newblock {\em Matchings and Covers in Hypergraphs}.
\newblock PhD thesis, University of Waterloo, 2016.

\bibitem{szestopalow17}
M.~Szestopalow.
\newblock Fractional ${K}_4$-covers.
\newblock {\em Graphs and Combinatorics}, 33(4):1055--1063, 2017.

\bibitem{turan41}
P.~Tur{\'a}n.
\newblock On an external problem in graph theory.
\newblock {\em Mat. Fiz. Lapok}, 48:436--452, 1941.

\bibitem{tuza90}
Z.~Tuza.
\newblock A conjecture on triangles of graphs.
\newblock {\em Graphs and Combinatorics}, 6(4):373--380, 1990.

\end{thebibliography}
\end{document}